\documentclass[11pt]{amsart}
\usepackage{amsmath, amssymb, latexsym, bm}

\renewcommand{\baselinestretch}{\baselinestretch}
\renewcommand{\baselinestretch}{1.1}
\numberwithin{equation}{section}

\newcommand{\Q}{\mathbb{Q}}
\newcommand{\R}{\mathbb{R}}
\newcommand{\Z}{\mathbb{Z}}

\renewcommand{\O}{\mathcal{O}}

\newtheorem{thm}{Theorem}[section]
\newtheorem{prop}[thm]{Proposition}
\newtheorem{lem}[thm]{Lemma}
\newtheorem{cor}[thm]{Corollary}

\theoremstyle{definition}
\newtheorem{defn}[thm]{Definition}

\theoremstyle{remark}
\newtheorem{rmk}[thm]{Remark}

\newcommand{\oc}{\overline{c}}
\newcommand{\od}{\overline{\alpha}}

\numberwithin{equation}{section}

\newcommand{\bx}{\bm x}
\newcommand{\by}{\bm y}

\newcommand{\bu}{\bm u}
\newcommand{\bv}{\bm v}

\newcommand{\pmu}{{(\mu)}}
\newcommand{\pnu}{{(\nu)}}

\newcommand{\Norm}{\mathbb{N}}
\newcommand{\NormK}{\Norm_{K/\Q}}
\newcommand{\Trace}{{\mathbb T}{\mathrm r}}
\newcommand{\TraceK}{\Trace_{K/\Q}}
\newcommand{\U}{\bm U}
\newcommand{\T}{\bm T}
\newcommand{\diag}{\textnormal{diag}}
\newcommand{\0}{{\bm 0}}

\newcommand{\rep}{\rightarrow \!\!\!-}

\newcommand{\p}{{\mathfrak p}}
\newcommand{\lcrown}{\vert\!\lceil}
\newcommand{\rcrown}{\rceil\!\vert}
\newcommand{\lbase}{\vert\!\lfloor}
\newcommand{\rbase}{\rfloor\!\vert}

\begin{document}

\title[Reduction and Waring's problem for quadratic forms]{Hermite reduction and a Waring's problem for integral quadratic forms over number fields}

\author{Wai Kiu Chan}
\address{Department of Mathematics and Computer Science, Wesleyan University, Middletown CT, 06459, USA}
\email{wkchan@wesleyan.edu}

\author{Mar\'ia In\'es Icaza}
\address{Instituto de Matem\'atica y F\'isica, Universidad de Talca, Casilla 747, Talca, Chile.}
\email{icazap@inst-mat.utalca.cl}

\subjclass[2010]{Primary 11E12, 11E25, 11E39}

\keywords{Waring's problem, Integral quadratic forms, Sums of squares, reduction theory}

\dedicatory{In memory of John Hsia, a mentor and a friend, who taught us everything we know about quadratic forms.}

\begin{abstract}
We generalize the Hermite-Korkin-Zolotarev (HKZ) reduction theory of positive definite quadratic forms over $\mathbb Q$ and its balanced version introduced recently by Beli-Chan-Icaza-Liu to positive definite quadratic forms over a totally real number field $K$.  We apply the balanced HKZ-reduction theory to study the growth of the {\em $g$-invariants} of the ring of integers of $K$.  More precisely, for each positive integer $n$, let $\O$ be the ring of integers of $K$ and $g_{\O}(n)$ be the smallest integer such that every sum of squares of $n$-ary $\O$-linear forms must be a sum of $g_{\O}(n)$ squares of $n$-ary $\O$-linear forms.  We show that when $K$ has class number 1, the growth of $g_{\O}(n)$ is at most an exponential of $\sqrt{n}$.  This extends the recent result obtained by Beli-Chan-Icaza-Liu on the growth of $g_{\mathbb Z}(n)$ and gives the first sub-exponential upper bound for $g_{\O}(n)$ for rings of integers $\O$ other than $\mathbb Z$.

\end{abstract}

\maketitle

\section{Introduction}

The question of determining which integers are sums of squares is a classical one dating back to the well-known work of Fermat, Euler, Legendre, and Lagrange.  A consequence of their work on sum of squares is that the Pythagoras number of the ring of integers $\Z$ is 4.   In their investigation of the Pythagoras numbers of affine algebras over a commutative ring $\O$, the authors of \cite{CDLR} show that these Pythagoras numbers are closely related to another arithmetic invariant $g_\O(n)$ of $\O$, the smallest number ($\leq \infty$) such that any sum of squares of $n$-ary $\O$-linear forms can be written as a sum of at most $g_\O(n)$ squares of such forms.  In the case over a global field $K$, it follows from the Hasse-Minkowski Principle \cite[Theorem 66:3]{OM} and the theory of representations of quadratic forms over local fields \cite[Theorem 63:21]{OM} that $g_K(n) = n + 3$ for all $n \geq 1$.  For more results on $g_K(n)$ when $K$ is an arbitrary field, the readers are referred to \cite{BLOP86} and \cite{CDLR}.

When $\O$ is a ring, very little work has been done on getting qualitative and quantitative results on $g_\O(n)$ before \cite{CDLR}, even in the case $\O = \Z$.  The earliest work in this direction may be due to Mordell \cite{Mordell30} and Ko \cite{Ko37} whose results together show that $g_\Z(n) = n + 3$ when $1 \leq n \leq 5$.  This can now be understood through the theory of integral representations of quadratic forms over local fields \cite{OM2, Riehm} and the fact that the class number of $I_m$, the integral quadratic form of sum of $m$ squares, is 1 when $m \leq 8$.  This local-to-global argument can be extended to show that $g_\O(n) = n + 3$ for all $n \geq 1$ when $\O$ is the ring of integers of a non-totally real number field, followed from the theory of spinor genus.  In light of these result one could ask if $g_\O(n)$ is always equal to $n + 3$.  In the case $\O = \Z$, this question was posed by Mordell in \cite{Mordell30} as a ``{\em new Waring's problem}".   Almost sixty years after Mordell posed his question, a negative answer is obtained by Kim-Oh in \cite{Kim-Oh97} where they show that $g_\Z(6) = 10$.  This is the last known exact value of $g_\O(n)$ thus far.  Nonetheless, the determination of $g_\O(n)$ remains an open interesting problem.

Since \cite{CDLR} there have been a lot of work devoted to obtaining upper bounds of $g_\O(n)$ when $\O$ is the ring of integers of a totally real number field.  Most notable is \cite{Icaza96} in which the second author of this paper shows that $g_\O(n)$ is always finite by finding an explicit upper bound on $g_\O(n)$ which is at least in the order of $n^{n^2}$ in the special case $\O = \Z$ already.  An exponential upper bound $g_\Z(n) = O(3^{n/2}n\, \log n)$ is obtained by Kim-Oh later in \cite{Kim-Oh05}.  A much better upper bound $g_\Z(n) = O(e^{(4 + 2\sqrt{2} + \epsilon)\sqrt{n}})$ for every $\epsilon > 0$ is obtained by us and our collaborators recently in \cite{BCIL}. The goal of this paper is to extend this result to rings of integers in totally real number fields of class number 1.

\begin{thm} \label{maintheorem}
Let $\O$ be the ring of integers of a totally real number field $K$ of class number $1$.  There exists constants $\kappa$ and $D$, depending only on $K$, such that
$$g_\O(n) \leq D\,e^{\kappa \sqrt{n}}.$$
\end{thm}

One of the new ideas introduced in \cite{BCIL} is the {\em balanced Hermite-Korkin-Zolotarev (HKZ) reduction} of positive definite quadratic forms over $\Q$, which is a modification of the classical HKZ-reduction.  For our purpose in this paper, we need a reduction theory which puts a positive definite quadratic form $Q$ in $n$ variables over a totally real number field in a reduced form whose Gram matrix is $U^tHU$ with $U$ upper triangular unipotent, $H = \diag(h_1, \ldots, h_n)$ with $\min(Q) = h_1$ and $h_i/h_{i+1}$ bounded (i.e. bounded with all its conjugates by a constant depending only on $K$ and $n$).  The most well-known reduction theory of positive definite quadratic forms over number fields is due to Humbert \cite{Humbert} which is essentially the same as the one by A. Weil given in his lectures at the University of Chicago \cite{Weil}.   This reduction theory was later generalized by  Koecher \cite{ko} and became a special case of his reduction theory of  ``Positivit\"{a}tsbereichen".  All of these follow the main idea of Minkowski's reduction theory over $\mathbb Q$ which relies on the (successive) minima of the function $\TraceK(Q(\bx))$, where $\TraceK$ is the trace from the totally real field $K$ to $\mathbb Q$.  This new function is a positive definite quadratic form over $\mathbb Q$ and has only finitely many minimal vectors in $\mathbb Z^{[K : \mathbb Q]n}$. Hence in principle\footnote{Of course, this is still difficult in practice as finding the minimal vectors of a positive definite quadratic form is one of the most important and difficult problems in the arithmetic theory of quadratic forms and lattice-based cryptography.} one can find all these minimal vectors by a finite search.  However, when viewed as vectors in $\mathcal O^n$, these minimal vectors of $\TraceK(Q(\bx))$ may not be extended to a basis for $\mathcal O^n$.  Consequently,  even in the case when $K$ has class number 1, Humbert reduction only guarantees that $Q$ represents, but {\em is not necessary integrally equivalent to}, a reduced form with the entries of its Gram matrix satisfying the desired inequalities.

Our approach of reduction theory follows the classical HKZ-reduction theory.  Instead of using the trace we use the norm function $\NormK$ from $K$ to $\mathbb Q$.  The function $\NormK(Q(\bx))$ is no longer a quadratic form over $\mathbb Q$ and finding its minimal vectors in $\mathcal O^n$ is much more difficult.  Adding to the problem is that when $K$ is not $\mathbb Q$ there are infinitely many such minimal vectors due to the infinitude of the units in $\mathcal O$.  So, some care will be taken in choosing the right minimal vectors in our construction of the HKZ-reduced forms.  Nonetheless, when $K$ has class number 1, every minimal vector of $\NormK(Q(\bx))$ can be extended to a basis for $\mathcal O^n$, and  as a result  every positive definite quadratic form will be indeed integrally equivalent to a HKZ-reduced form in this special case.

The rest of the paper is organized as follows.  In Section 2 notations and terminologies that are used in the rest of the paper will be introduced.  In Section 3 we present the HKZ-reduction of a positive definite quadratic form and its balanced version  over a totally real number field.   Section 4 contains a few technical lemmas regarding the representations of positive semidefinite quadratic forms by a sum of small number of squares.  The proof of Theorem \ref{maintheorem} will be in Section 5, the last section of this paper.

\section{Notations and preliminaries}

Let $K$ be a totally real number field of degree $d$ over $\Q$ and $\O$ be its ring of integers.  The norm and trace from $K$ to $\Q$ are denoted by $\NormK$ and $\TraceK$, respectively.  The notation ``$\nu \mid \infty$" means that $\nu$ is an infinite place of $K$ which is an embedding of $K$ into $\R$.  The image of an element $a$ in $K$ under $\nu$ is denoted by $a^\pnu$.
Every infinite place $\nu$ of $K$ gives rise to an Archimedean valuation $\vert \,\,\vert_\nu$ on $K$ defined by $\vert a \vert_\nu: = \vert a^\pnu \vert$, where $\vert\,\,\vert$ is the usual absolute value on $\R$.

We will write $a \succ 0$ if $a$ is a totally positive element in $K$, i.e. $a^\pnu > 0$ for all $\nu \mid \infty$.  If $a$ and $b$ are in $K$, then $a \succ b$ means $a - b \succ 0$.  For any $a \in K$, let
$$ \lbase a \rbase: = \min \{\vert a \vert_\nu : \nu \mid \infty\} \quad \mbox{ and } \quad \lcrown a \rcrown : = \max\{\vert a \vert_\nu : \nu \mid \infty \}.$$

Let $K_\infty$ be the product $\prod_{\nu \mid \infty}K_\nu \cong \R^d$, where $K_\nu$ is a completion of $K$ with respect to $\vert \,\,\vert_\nu$.  Let
\begin{equation} \label{beta}
\beta := \inf \{ r > 0 : \forall (\alpha_\nu) \in K_\infty,  \exists a \in \O \mbox{ such that } \vert \alpha_\nu - a \vert_\nu < r,\, \forall \nu \mid \infty\}
\end{equation}
which is an invariant of $K$.  This $\beta$ exists and is a finite number because the image of $\O$ under the ring monomorphism $a \longmapsto (a^\pnu)$ is a full-rank $\Z$-lattice in $K_\infty$.  It is shown in \cite[Theorem 6]{OV} (see also the discussion thereafter) that $\beta \leq \frac{1}{2} \sqrt{\vert d_K \vert}$, where $d_K$ is the discriminant of $K$.  Let
\begin{equation} \label{cbeta}
{\mathcal B}_K: =  \{(\alpha_\nu) \in K_\infty : \vert \alpha_\nu \vert_\nu \leq \beta \mbox{ for all } \nu \mid \infty \}.
\end{equation}
It is clear that for every $\alpha \in K_\infty$, there exists $a \in \O$ such that $\alpha - a \in {\mathcal B}_K$.

Let $Q(x_1, \ldots, x_n)$ be a quadratic form in $n$ variables over $K$.  We always identify $Q$ with its Gram matrix.  By completing squares, we can write
$$Q(x_1, \ldots, x_n) = \sum_{i=1}^n h_i\left(x_i + \sum_{j=i+1}^n u_{ij}x_j\right)^{\!\!\!2}.$$
This is called the {\em Lagrange expansion} of $Q$.  The $h_i$ are called the outer coefficients of $Q$ and the $u_{ij}$ are called the inner coefficients of $Q$.  The Gram matrix of $Q$ can be written as $U^tHU$, where $H = \diag(h_1, \ldots, h_n)$ and $U$ is the unipotent upper triangular matrix whose entries above the main diagonal are the outer coefficients $u_{ij}$.

A quadratic form over $K$ is called integral if its Gram matrix has entries from $\O$.  The scale of a quadratic form over $K$ is the fractional ideal of $K$ generated by the entries of the Gram matrix of that quadratic form.   Let $A$ and $B$ be two integral quadratic forms over $K$ in $n$ and $m$ variables, respectively.  We say that $A$ is represented by $B$, written $A \rep B$, if there exists an $m\times n$ matrix $U$  over $\O$ such that $A = B[U]: = U^tBU$.  If $m = n$ and $U$ is in $\text{GL}_n(\O)$, then $A$ and $B$ are said to be integrally equivalent.

For any positive integer $m$, $I_m$ denotes the quadratic form of sum of $m$ squares.  It is clear that an integral quadratic form is represented by $I_m$ if and only if it can be written as a sum of at most $m$ squares of linear forms with coefficients from $\O$.

For each $\nu \mid \infty$, let $Q^\pnu$ be the quadratic form over $K_\nu$ obtained by applying $\nu$ to the coefficients of $Q$.  We call $Q$ positive definite if $Q^\pnu$ is positive definite for every $\nu$.  For any positive definite quadratic form $Q$ in $n$ variables over $K$, its minimum, denoted $\min(Q)$, is defined by
$$\min(Q): = \min\{\NormK(Q(\bx)): \0 \neq \bx \in \O^n\}.$$
A vector $\bx$ in $\O^n$ is called a minimal vector of $Q$ if $Q(\bx) = \min(Q)$.

For every quadratic form $Q$ over $K$, let $\det(Q)$ be the determinant of its Gram matrix and set $d(Q): = \NormK(\det(Q))$.  Let
$$\gamma_{n, K} := \sup_Q \frac{\min(Q)}{d(Q)^{\frac{1}{n}}},$$
where $Q$ runs over all positive definite quadratic forms in $n$ variables over $K$.  By \cite[Theorem 1]{Icaza97}, we have
$\gamma_{n,K}\leq\sigma_{n,K}$, with
$$\sigma_{n,K} = 4^d\omega_n^{-\frac{2d}{n}}\vert d_K \vert$$
where $\omega_n$ is the volume of the $n$-dimensional unit sphere.  It is well-known that
$$\omega_{n} = \pi^{\frac{n}{2}}\Gamma\left(\frac{n}{2}+1\right)^{-1} = \begin{cases}
\frac{\pi^{\frac{n}{2}}}{\left(\frac{n}{2}\right)!} & \mbox{ if $n$ is even},\\
    & \\
\frac{\pi^{\frac{n-1}{2}}\, 2^{n+1}\, \left(\frac{n+1}{2}\right)!}{(n+1)!} & \mbox{ if $n$ is odd}.
\end{cases}$$
It follows from standard estimates of $n!$ (see \cite{BCIL}, for example) that
\begin{equation}\label{hermiteconstant}
\sigma_{n, K} \leq \left(e^{-1 + \frac{1}{n}}\,n^{1 + \frac{1}{n}}\right)^d\, \vert d_K\vert, \quad \mbox{ for all $n \geq 1$}.
\end{equation}

Since $K$ will be clear from the context, we simply write $\sigma_n$ instead of $\sigma_{n,K}$.  Then,
for any positive definite quadratic form $Q$ in $n$ variables over $K$, we have
\begin{equation*}
\min(Q) \leq \sigma_n \, d(Q)^{\frac{1}{n}}.
\end{equation*}
For any positive integer $m$, let
\begin{equation*}
\alpha(m):=\sigma_{m+1}\prod_{k=2}^{m+1}\sigma_k^{\frac{1}{k-1}}. \nonumber
\end{equation*}
Following the proof of \cite[Lemma 3.5]{BCIL} and using \eqref{hermiteconstant}, we obtain
\begin{equation*} 
\alpha(m) \leq D_1 \, \vert d_K \vert^{1 + \Sigma(m)}\, e^{\frac{d}{2}(\ln m)^2}= : \od(m),
\end{equation*}
where $D_1$ is a constant depending only on $d$ and $\Sigma(m) = 1 + \frac{1}{2} + \cdots + \frac{1}{m}$.  It is easy to see that $\od(m)$ is an increasing function of $m$.  We choose $D_1$ so that $\od(m) \geq 1$ for all positive integers $m$.  Since $\Sigma(m) \leq 1 + \ln(m)$ for all $m \geq 1$, we have
\begin{equation} \label{alpha_m_bar}
\od(m) \leq D_2\, e^{\sqrt{m}}
\end{equation}
for some constant $D_2$ depending only on $K$.

In later discussion it will be convenient and sometimes necessary to expand our discussion from quadratic forms to Humbert forms.  An $n\times n$ Humbert form over $K$ is a symmetric matrix in $M_n(K_\infty) = \bigoplus_{\nu \mid \infty} M_n(K_\nu) \cong M_n(\R)^d$.  If $S$ is a Humbert form over $K$, $S^\nu$ denotes its $\nu$-th component and we write $S = (S^\nu)$.  If $\alpha$ is a real number and $S$ is a Humbert form, then $\alpha S$ is the Humbert form with $\alpha S^\nu$ as its $\nu$-th component.

A quadratic form $Q$ over $K$ in $n$ variables can be viewed as an $n\times n$ Humbert form whose $\nu$-th component is simply $Q^\pnu$.  A Humbert form is said to be positive definite if all its  components are positive definite.  One advantage of using Humbert forms is that we can take the positive definite square root of any diagonal matrix $H = \diag(h_1, \ldots, h_n)$ with totally positive diagonal entries from $K$: it is simply the Humbert form $\sqrt{H}$ with $\sqrt{H}^\nu = \diag\left(\sqrt{h_1^\pnu}, \ldots, \sqrt{h_n^\pnu}\right)$ for every $\nu \mid \infty$.

If $S$ is an $n\times n$ Humbert form over $K$ and $X \in M_n(K_\infty)$, then $S[X]$ is the Humbert form $X^tSX$.  Two $n\times n$ Humbert forms $S_1$ and $S_2$ over $K$ are {\em integrally equivalent} if $S_2 = S_1[T]$ for some $T \in \text{GL}_n(\O)$.  This generalizes the classical notion of integrally equivalence of quadratic forms over $K$.

\section{Reduction of quadratic forms}


Classically, a positive definite real quadratic form $Q$ in $n$ variables is called HKZ-reduced if its inner and outer coefficients satisfy
$$h_1 = \min(Q), \quad 0 < h_i < \frac{4}{3}h_{i+1}, \quad 1 \leq i \leq n-1$$
and
$$\vert u_{ij} \vert \leq \frac{1}{2}, \quad 1 \leq i, j \leq n.$$
The classical HKZ-reduction theory says that every positive definite real quadratic form is integrally equivalent to a HKZ-reduced form.

In this section, we will present a generalization of HKZ-reduction for positive definite quadratic forms over a totally real number field $K$ of degree $d$ over $\mathbb Q$.  We will also generalize the notion of a balanced HKZ-reduced form introduced recently in \cite{BCIL}.

\subsection{HKZ-reduced forms--definitions}

We begin this subsection by stating a couple of technical lemmas.

\begin{lem} \label{tracenorm}
There exists a constant $D_3$, depending only on $K$, such that for every totally positive element $h$ of $K$, there exists $z \in \O\setminus\{0\}$ with
$$\TraceK(hz^2) \leq D_3\, \NormK(h)^{\frac{1}{d}}.$$
\end{lem}
\begin{proof}
We identify $\O$ as a $\Z$-lattice on $K_\infty \cong \R^d$  of volume $\sqrt{\vert d_K \vert}$ through the embedding $a \mapsto (a^\pnu)$.  Minkowski's theorem on convex bodies implies that any $\0$-symmetric compact convex body in $K_\infty \cong \R^d$ of volume at least $2^d\sqrt{\vert d_K\vert}$ must contain a nonzero element of $\O$.  This applies to ellipsoids of the type $\{\bx \in \R^d: \sum_{i=1}^d a_ix_i^2 \leq r\}$ where $a_1, \ldots, a_d$ are positive real numbers such that $a_1\cdots a_d = 1$.   So, there must be a positive real number $D_3$, depending only on $K$, and a nonzero element $z$ of $\O$ such that
$$\sum_{\nu \mid \infty} \frac{h^\pnu}{\NormK(h)^{1/d}} \,(z^2)^\pnu \leq D_3$$
which is what we need to show.
\end{proof}

\begin{lem} \label{reduced_elements}
There exists a constant $\lambda$, depending only on $K$, with the following property: for every totally positive element $h$ of $\O$ there is a unit $\epsilon$ of $\O$ such that
\begin{equation*} 
(\epsilon^2h)^\pnu \leq \lambda\, (\epsilon^2h)^\pmu
\end{equation*}
for any two real embeddings $\mu, \nu$ of $K$.
\end{lem}
\begin{proof}
Let $h$ be a totally positive element in $\O$.  Let $y_0$ be a nonzero element in $\O$ such that
$$\TraceK(hy_0^2) = \min_{y \in \O\setminus\{0\}} \TraceK(hy^2).$$
For any $\mu \mid \infty$, by \cite[33:8]{OM} there exists a unit $\epsilon_\mu$ such that
$$\vert \epsilon_\mu \vert_\mu > 1, \quad \vert \epsilon_\mu\vert_\sigma< 1 \quad \mbox{ for all } \sigma \neq \mu.$$
Then, since $\TraceK(hy_0^2) \leq \TraceK(h(y_0\epsilon_\mu)^2)$, we have
$$(hy_0^2)^\pmu(\vert \epsilon_\mu \vert_\mu^2 - 1)  + \sum_{\sigma \neq \nu, \mu}(hy_0^2)^{(\sigma)} (\vert \epsilon_\mu \vert_\sigma^2 - 1) \geq (hy_0^2)^\pnu (1 - \vert \epsilon_\mu \vert_\nu^2).$$
On the left-hand side of this inequality, every term within the summation sign is negative.  Thus,
$$(hy_0^2)^\pnu \leq C\, (hy_0^2)^\pmu, \quad \mbox{ where } C: = \max_\mu \max_{\nu \neq \mu} \frac{\vert \epsilon_\mu \vert_\mu^2- 1}{1 - \vert \epsilon_\mu\vert_\nu^2}.$$

Now, for every $x \in \O\setminus \{0\}$,
$$\TraceK(hy_0^2) \leq \TraceK(hx^2) \leq \TraceK(hy_0^2)\, \TraceK((y_0^{-1}x)^2)$$
because $hy_0^2$ and $(y_0^{-1}x)^2$ are totally positive.  Therefore, $\TraceK((y_0^{-1}x)^2)$ is at least 1 for all $x \in \O\setminus \{0\}$.  By Lemma \ref{tracenorm}, there must be a nonzero element $x_0$ of $\O$ such that
$$1 \leq \TraceK((y_0^{-1}x_0)^2)  \leq D_3\, \NormK(y_0^{-2})^{\frac{1}{d}},$$
which means that $\vert \NormK(y_0)\vert \leq D_3^{\frac{d}{2}}$.  Hence, there is a finite subset $\mathcal Y$ of $\O$, depending only on $K$, such that $y_0$ is an associate of an element in $\mathcal Y$.  Let
\begin{equation} \label{lambda}
\lambda: = C\, \left(\max_{y \in \mathcal Y}\max_{\nu, \mu} \frac{y^\pnu}{y^\pmu}\right)^2.
\end{equation}
This $\lambda$ is a constant depending only on $K$.  Let $\epsilon$ be a unit of $\O$ such that $\epsilon^{-1} y_0 \in \mathcal Y$.  Then,
$$(\epsilon^2 h (y_0\epsilon^{-1})^2)^\pnu \leq C (\epsilon^2 h (y_0\epsilon^{-1})^2)^\pmu,$$
and therefore,
\begin{equation*}
(\epsilon^2h)^\pnu \leq \lambda (\epsilon^2 h)^\pmu
\end{equation*}
which is what we need to prove.
\end{proof}

\begin{defn} \label{weakly-reduced-definition}
Let $K$ be a totally real number field.  A positive definite quadratic form in $n$ variables over $K$ is said to be {\em weakly reduced} if its outer coefficients $h_1, \ldots, h_n$ satisfy:
\begin{enumerate}
\item $\NormK(h_1) = \min(Q)$ and  $\NormK(h_i) \leq \od(j-i) \NormK(h_j)$ for any $1\leq i < j \leq n$,

\item $h_i^{\pnu} \leq \lambda h_i^{\pmu}$ for any  $1 \leq i \leq n \mbox{ and } \mu, \nu \mid \infty$.
\end{enumerate}
Here $\od$ and $\lambda$ are defined in \eqref{alpha_m_bar} and \eqref{lambda} respectively.
\end{defn}

\begin{rmk}
In the case $K = \mathbb Q$, the definition of weakly reduced quadratic forms presented above is slightly different from the one given in \cite[Definition 3.2]{BCIL}.  However, it is clear from \cite[Lemma 3.4]{BCIL} that these two definitions are indeed the same in this special case.
\end{rmk}

\begin{lem}
The outer coefficients of a weakly reduced quadratic form in $n$ variables over $K$ satisfy the inequalities
\begin{equation} \label{hihj2}
h_i^\pnu \leq \lambda^2\, \od(j-i)^{\frac{1}{d}}\, h_j^\pmu
\end{equation}
for any $1 \leq i < j \leq n$ and $\nu, \mu \mid \infty$.
\end{lem}
\begin{proof}
It follows from condition (1) in Definition \ref{weakly-reduced-definition} that there must be a $\nu_0 \mid \infty$ such that
$$h_i^{(\nu_0)} \leq \od(j-i)^{\frac{1}{d}}h_j^{(\nu_0)}.$$
Then, for all $\nu$ and $\mu$,
\begin{eqnarray*}
h_i^\pnu & \leq & \lambda\, h_i^{(\nu_0)}\\
    & \leq & \lambda\, \od(j-i)^{\frac{1}{d}}\, h_j^{(\nu_0)}  \\
    & \leq & \lambda^2\, \od(j-i)^{\frac{1}{d}}\, h_j^\pmu
\end{eqnarray*}
as claimed.
\end{proof}

In later discussion it will be more convenient to use inequalities weaker than \eqref{hihj2} but valid for the case $i = j$.

\begin{cor} \label{hihj3}
The outer coefficients of a weakly reduced quadratic form in $n$ variables over $K$ satisfy the inequalities
$$h_i^\pnu \leq \lambda^2\, \od(n)^{\frac{1}{d}}\, h_j^\pmu$$
for $1 \leq i \leq j \leq n$ and $\nu, \mu \mid \infty$.
\end{cor}
\begin{proof}
This is clear when $i < j$ as $\od$ is a nondecreasing function on the positive integers.  It follows from Lemma \ref{reduced_elements} that the constant $\lambda$ must be at least 1.  The quantity $\od(n)$ is also at least 1 by definition.  So, by Definition \ref{weakly-reduced-definition},
$$h_i^\pnu \leq \lambda\, h_i^\pmu \leq \lambda^2\, \od(n)^{\frac{1}{d}}\, h_i^\pmu$$
for any $\nu, \mu \mid \infty$.
\end{proof}

The sole purpose of introducing the weakly reduced quadratic forms is to control the growth of the outer coefficients.  As is done in the classical case,  the reduction can go a step further and arrive at a quadratic form whose inner coefficients are bounded.

\begin{defn} \label{hkz-definition}
Let $K$ be a totally real number field.  A positive definite quadratic form in $n$ variables over $K$ is called {\em HKZ-reduced} if it is weakly reduced and its inner coefficients are lying inside the set $\mathcal B_K$ defined in \eqref{cbeta}.
\end{defn}

\begin{prop} \label{hermite-inequality}
There are constants $\delta, C_1, \ldots, C_n$, where $\delta$ depends only on $n$ and $K$ and each $C_j$ depends only on $j$ and $K$, such that if $Q = [a_{ij}]$ is a HKZ-reduced quadratic form in $n$ variables over $K$, then for any $j = 1, \ldots, n$ and $\nu \mid \infty$,
$$h_j^\pnu \leq a_{jj}^\pnu \leq C_j\,h_j^\pnu \quad \mbox{ and } \quad \det(Q)^\pnu \leq a_{11}^\pnu\cdots a_{nn}^\pnu \leq \delta\, \det(Q)^\pnu.$$
\end{prop}
\begin{proof}
Let $Q = [a_{ij}]$ be a HKZ-reduced quadratic form in $n$ variables over $K$.  It follows from the definition that for $j = 1, \ldots, n$,
$$a_{jj} = h_j + \sum_{i=1}^{j-1} h_i u_{ij}^2.$$
Since each $h_i$ is totally positive, therefore $h_j^\pnu \leq a_{jj}^\pnu$.

It is clear that we may take $C_1$ to be 1 as $h_1 = a_{11}$ by definition. For $j = 2, \ldots, n$, from the definition of the set $\mathcal B_K$ and \eqref{hihj2}, we have
$$a_{jj}^\pnu \leq h_j^\pnu + \beta^2\,\lambda^2\,\sum_{i=1}^{j-1} \od(j-i)^{\frac{1}{d}}\,h_j^\pnu.$$
Since the function $\od$ is nondecreasing, we could take $C_j$ to be $1 + \beta^2\,\lambda^2\,\od(j-1)^{\frac{1}{d}}\,(j-1)$.  Now, let $\delta = C_1\cdots C_n$.  It is clear that $h_1\cdots h_n = \det(Q)$.  Thus,
$$h_1^\pnu \cdots h_n^\pnu \leq a_{11}^\pnu \cdots a_{nn}^\pnu \leq \delta\, \det(Q)^\pnu.$$
\end{proof}

\subsection{HKZ-reduced forms--existence}

In this subsection we will conduct the discussion in the language of quadratic spaces and lattices.  Unless specified otherwise, the quadratic map on any quadratic space over $K$ will be simply denoted by $Q$.  We use $B$ to denote the associated bilinear form defined by $B(\bx, \by) = \frac{1}{2}(Q(\bx + \by) - Q(\bx) - Q(\by))$. If $\Lambda$ is a free $\O$-lattice with a quadratic map $Q$, then the Gram matrix with respect to a basis $\bu_1, \ldots, \bu_n$ is the matrix $(B(\bu_i, \bu_j))$, which is the Gram matrix of a quadratic form associated to $\Lambda$.  Different bases for $\Lambda$ yield integrally equivalent quadratic forms, and this sets up an one-to-one correspondence between isometry classes of free $\mathcal O$-lattices and integral equivalence classes of quadratic forms over $K$.  For more on this correspondence and discussion of quadratic forms from this geometric perspective, the readers may consult \cite{OM}.

For any nonzero vector $\bu$ in the underlying space of an $\O$-lattice $\Lambda$, its coefficient ideal relative to $\Lambda$ is the set
$$\mathfrak a_{\bu}(\Lambda): = \{ \alpha \in K : \alpha \bu \in \Lambda \}.$$
It is a fractional ideal of $K$, and $K\bu \cap \Lambda = \mathfrak a_{\bu}(\Lambda)\bu$.  Thus, the annihilator of $K\bu \cap \Lambda$ modulo $\O\bu$ is $\mathfrak a_{\bu}(\Lambda)^{-1}$.  If, in addition, $\bu$ is in $\Lambda$, then $\mathfrak a_{\bu}(\Lambda)$ is the inverse of an integral ideal.   A vector $\bu$ in $\Lambda$ is called {\em unimodular} (or {\em maximal} in \cite{OM} or primitive in some literature) if $\mathfrak a_{\bu}(\Lambda)$ is $\O$.  In that case, $\O\bu$ is a direct summand of $\Lambda$.

\begin{lem} \label{build_basis}
Let $\bu_1$ be a nonzero vector in an $\O$-lattice $\Lambda$, and $\pi: K\Lambda \longrightarrow K\Lambda$ be the orthogonal projection on the orthogonal complement of $K\bu_1$.  Suppose that $\bv_2, \ldots, \bv_n$ are linearly independent vectors in $\pi(\Lambda)$.  Let $\bu_2, \ldots, \bu_n$ be vectors in $\Lambda$ such that $\pi(\bu_j) = \bv_j$ for all $j \geq 2$.
\begin{enumerate}
\item The vectors $\bu_1, \bu_2, \ldots, \bu_n$ are linearly independent.

\item If $\bu_1$ is a unimodular vector in $\Lambda$ and $\{\bv_2, \ldots, \bv_n\}$ is a basis for $\pi(\Lambda)$, then $\{\bu_1, \bu_2, \ldots, \bu_n\}$ is a basis for $\Lambda$.

\item For $j = 2, \ldots, n$, $\mathfrak a_{\bu_j}(\Lambda) \subseteq \mathfrak a_{\bv_j}(\pi(\Lambda))$.

\item If $\mathfrak a_{\bu_1}(\Lambda) = \mathfrak b^{-1}$ and $\mathfrak b'$ is the annihilator of $\pi(\Lambda)$ modulo $\O\bv_2 \oplus \cdots \oplus \O\bv_n$, then $\mathfrak b \mathfrak b'$ is contained in the annihilator of $\Lambda$ modulo $\O\bu_1 \oplus \cdots \oplus\O\bu_n$.
\end{enumerate}
\end{lem}
\begin{proof}
Part (1) is straightforward and part (2) is essentially the same as \cite[Lemma 3.1]{BCIL}.  For part (3), suppose that $\alpha \bu_j \subseteq \Lambda$.  Then $\alpha \bv_j = \alpha \pi(\bu_j) \in \pi(\Lambda)$ and the assertion follows immediately.

As for part (4), let $\bx$ be a vector in $\Lambda$.  Then $\bx = \alpha_1\bu_1 + \alpha_2\bu_2 + \cdots + \alpha_n\bu_n$, where $\alpha_1, \ldots, \alpha_n \in K$.  For any $b' \in \mathfrak b'$,
$$b'\pi(\bx)= b'\alpha_2\bv_2 + \cdots + b'\alpha_n\bv_n$$
which must be a vector in $\O\bv_2 \oplus \cdots \oplus \O\bv_n$.  Thus, $b'\alpha_i \in \O$ for all $i \geq 2$ and hence $b'\alpha_1\bu_1 \in \Lambda$.  Thus, $b'\alpha_1 \in \mathfrak b^{-1}$; so for any $b \in \mathfrak b \subseteq \O$, we have $bb' \alpha_1 \in \O$ and
$$bb'\bx = bb'\alpha_1\bu_1 + bb'\alpha_2\bu_2 + \cdots + bb'\alpha_n\bu_n$$
is a vector in $\O\bu_1 \oplus \cdots \oplus \O\bu_n$.  This proves part (3).
\end{proof}

Let $\Lambda$ be a positive definite $\O$-lattice of rank $n$.  The minimum of $\Lambda$ is defined as
$$\min(\Lambda): = \min\{\NormK(Q(\bx)) : \0 \neq \bx \in \Lambda \}.$$
If $\Lambda$ is free and $Q$ is the quadratic form associated to a particular basis for $\Lambda$, then $\min(\Lambda)$ is equal to $\min(Q)$ defined earlier.  A vector $\bv \in \Lambda$ is called a {\em minimal vector} if $\NormK(Q(\bv)) = \min(\Lambda)$.  The set of minimal vectors in $\Lambda$ is denoted by $\mathfrak M(\Lambda)$.  When $K$ has class number 1, every minimal vector of $\Lambda$ must be unimodular.  However, this may not be the case when the class number of $K$ is not 1.  An example can be found in \cite{BI}.

\begin{lem} \label{minimal_ideal}
Let $\bv$ be a minimal vector of a positive definite $\O$-lattice.  If $\mathfrak b^{-1}$ is the coefficient ideal of $\bv$ relative to $\Lambda$, then $\mathfrak b$ is an integral ideal of minimal norm in its class.
\end{lem}
\begin{proof}
Suppose that $\mathfrak b' = \lambda \mathfrak b$ is an integral ideal in the same class of $\mathfrak b$ whose norm is strictly smaller than the norm of $\mathfrak b$.  Then $\NormK(\lambda) < 1$.  But $\lambda$ is in $\mathfrak b^{-1}$ and hence $\lambda \bv$ is in $\Lambda$.  This is a contradiction since $\NormK(Q(\lambda \bv)) < \NormK(Q(\bv)) = \min(\Lambda)$.
\end{proof}

 A set of linearly independent vectors $\{\bu_1, \ldots, \bu_n\}$ is called {\em HKZ-reduced} if the associated Gram matrix $(B(\bu_i, \bu_j))$ is the Gram matrix of a HKZ-reduced quadratic form.

\begin{thm} \label{hkz-reduction-lattice}
Let $\Lambda$ be a positive definite $\O$-lattice of rank $n$.  There exists a free sublattice $L$ of $\Lambda$ such that $\min(L) = \min(\Lambda)$ and
\begin{enumerate}
\item $L$ has a HKZ-reduced basis;

\item $[\Lambda : L] \leq \left(\frac{d!}{d^d}\sqrt{\vert d_K\vert}\right)^{\!n^2}$;

\item $\Lambda = L$ if the class number of $K$ is $1$.
\end{enumerate}
\end{thm}
\begin{proof}
Let $\bv_1$ be a minimal vector of $\Lambda$. Using Lemma \ref{build_basis} and an induction argument, we obtain an orthogonal basis $\bv_1, \ldots, \bv_n$ of $K\Lambda$ and linearly independent vectors $\bu_1, \ldots, \bu_n$ in $\Lambda$ such that $\bu_1 = \bv_1$, $\bv_j \in \mathfrak M(\pi_j(\Lambda))$ for $2 \leq j \leq n$, where $\pi_j$ is the orthogonal projection of $K\Lambda$ onto the orthogonal complement of $\text{span}(\bv_1, \ldots, \bv_{j-1})$, and $\pi_j(\bu_j) = \bv_j$.
For $1 \leq j \leq n$, let $h_j = Q(\bv_j)$.  By Lemma \ref{reduced_elements}, we may assume that
\begin{equation*} \label{hkz2}
h_j^\pnu \leq \lambda h_j^\pmu \quad \mbox{ for all } \nu, \mu \mid \infty.
\end{equation*}

For $2\leq j \leq n$,
$$\bu_j = \bv_j + \sum_{\ell = 1}^{j-1} u_{\ell j}\bv_\ell, \quad u_{\ell j} \in K,$$
which can be rewritten as
$$\bu_j = \bv_j + u_{j-1, j} \bu_{j-1} + \sum_{\ell = 1}^{j-2}(u_{\ell j} - u_{j-1, j})\bv_\ell.$$
For any $i < j$, let $\alpha_{ij} \in \O$ such that $\alpha_{ij} - u_{ij} \in \mathcal B_K$.   Since $\pi_j(\bu_j) = \pi_j(\bu_j - \alpha_{j-1, j}\bu_{j-1})$, we may change $\bu_j$ to $\bu_j - \alpha_{j-1, j}\bu_{j-1}$ and assume that $u_{j-1, j}$ is already in $\mathcal B_K$ at the outset.  We may continue this process and at the end we may very well assume that all the $u_{ij}$ are in $\mathcal B_K$.

Let $L$ be the free $\O$-lattice $\O\bu_1 \oplus \cdots \oplus \O\bu_n$.  By Lemma \ref{build_basis}, $L = \Lambda$ when the class number of $K$ is 1.  It is clear that $\min(L) = \min(\Lambda)$, and $\bv_j \in \mathfrak M(\pi_j(L))$.  By Lemmas \ref{build_basis} and \ref{minimal_ideal} , the annihilator of $\Lambda$ modulo $L$ is of the form $\mathfrak b_1\cdots \mathfrak b_n$, where each $\mathfrak b_i$ is an integral ideal containing the ideal of the smallest norm in its class.  Therefore, by Minkowski's bound \cite[Page 166]{FT},
$$[\Lambda : L] \leq \NormK(\mathfrak b_1\cdots \mathfrak b_n)^n \leq \left(\frac{d!}{d^d}\sqrt{\vert d_K \vert}\right)^{\!n^2}.$$

Every vector $\bv$ in $L$ is of the form $\bv = \sum_{i=1}^n x_i \bu_i$, where $x_i \in \mathcal O$ for all $i$, which can be re-written as
\begin{equation*} \label{formofv}
\bv = \sum_{i=1}^n \left(x_i + \sum_{j = i+1}^n u_{ij}x_j\right)\bv_i.
\end{equation*}
Since $\bv_1, \ldots, \bv_n$ is an orthogonal basis for $KL$, we have
\begin{equation} \label{newform}
Q(\bv) = \sum_{i=1}^n h_i \,\left(x_i + \sum_{j = i + 1}^n u_{ij}x_j\right)^{\!\!2}
\end{equation}
which is a positive definite quadratic form over $K$ with $h_1, \ldots, h_n$ as its outer coefficients and $u_{ij}$ as its inner coefficients. Following the proofs of \cite[Lemma 3.4]{BCIL} and \cite[Lemma 2.4]{Schnorr}, one may deduce that
\begin{equation*} \label{hkz1}
\NormK(h_i) \leq \od(j - i)\NormK(h_j), \quad \mbox{ for all } 1 \leq i < j \leq n.
\end{equation*}

In sum, the quadratic form \eqref{newform} is HKZ-reduced and hence the basis $\bu_1, \ldots, \bu_n$ is a HKZ-reduced basis for $L$.
\end{proof}

An immediate corollary of this theorem is:

\begin{cor} \label{hkz-class-number-1}
Let $K$ be a totally real number field of class number $1$.  Then every positive definite quadratic form over $K$ is integrally equivalent to a HKZ-reduced form.
\end{cor}

When the class number of $K$ is bigger than 1, Theorem \ref{hkz-reduction-lattice} shows that every positive definite quadratic form $Q$ in $n$ variables over $K$ represents a HKZ-reduced form $\tilde{Q}$, i.e. there exists $X \in M_n(\O)$ such that $\tilde{Q} = X^t Q X$ and $\vert \NormK(\det X) \vert$ is bounded above by a constant depending only on $K$ and $n$.  By \cite[Page 279, Th\'{e}or\'{e}me 2]{Humbert}, there exists a finite set of matrices $\mathfrak T(K, n)$, which depends only on $K$ and $n$, with the property that for every such $X$ there are $U \in \text{GL}_n(\O)$ and $T \in \mathfrak T(K, n)$ such that $X = UT$.  Let $\tilde{\mathcal P}_n$ be the set of HKZ-reduced quadratic forms in $n$ variables. Then every positive definite quadratic form in $n$ variables over $K$ must be integrally equivalent to a quadratic form in the set
$$\mathcal Q_n: = \bigcup_{T \in \mathfrak T(K, n)} T^{-t}\, \tilde{\mathcal P}_n\, T^{-1}.$$
It is likely that different forms in $\mathcal Q_n$ are equivalent even if they are not on the boundary of $\mathcal Q_n$.  It would be interesting to determine a subset of $\mathcal Q_n$ which is a fundamental domain of the action of $\text{GL}_n(\O)$ on the set of positive definite quadratic forms in $n$ variables over $K$.

\subsection{Balanced HKZ-reduced forms}

For our purpose of obtaining the upper bound for $g_{\O}(n)$ in Theorem \ref{maintheorem}, we need to control the growth of the inner coefficients of both the quadratic form and its dual.  This leads to the definition of balanced HKZ-reduced forms which will be given in this subsection.

Let $\U$ be the group scheme of $n\times n$ upper triangular unipotent matrices.  For $1 \leq k \leq n$, let $\T_k$ be the additive group of $n\times n$ matrices spanned by $E_{1,1+k}, \ldots, E_{n-k, n}$, where $E_{ij}$ denotes the $n\times n$ matrix with 1 in the $(i,j)$-entry and 0 elsewhere.  For any subvariety $\bm V$ of $\text{GL}_n$ and any subset $X$ of a $\Z$-algebra, $\bm V(X)$ denotes the set of points in $\bm V$ with entries from $E$.

For any nonnegative integer $m$, let $c(m)$ be the coefficient of $x^m$ in the Maclaurin series of $\exp(\frac{\beta x}{1 - x})$.  Let
$$\mathcal E(m): = \{(\gamma_\nu) \in K_\infty : \vert \gamma_\nu \vert_\nu \leq c(m) \mbox{ for all } \nu \mid \infty\}.$$

\begin{lem}\textnormal{\cite[Lemma 4.4]{BCIL}} \label{cm}
There exists a constant $D_4$, depending only on $K$, such that
$$c(m)\leq \oc(m): =  D_4\,e^{2\sqrt{\beta m}}$$
for any $m \geq 1$.
\end{lem}

Note that $\oc(m)$ is an increasing function of $m$.  We will choose $D_4 \geq 1$ so that $\oc(m) \geq 1$ for all $m \geq 0$.

\begin{lem}\label{xy}
For any $X\in \U(K_\infty)$, there exists  $Y\in \U(\O)$ such that:
\begin{enumerate}
\item[(1)] $XY$ can be written as $XY=\exp(Z_1)\cdots \exp(Z_{n-1})$ where $Z_k\in \T_k({\mathcal B}_K)$ for $1 \leq k \leq n-1$,

\item[(2)] the $(i,j)$-entries of both $XY$ and $(XY)^{-1}$ are in $\mathcal E(j - i)$ for $1 \leq i < j \leq n$.
\end{enumerate}
\end{lem}
\begin{proof}
The proof of \cite[Lemma 4.1]{BCIL} carries over verbatim and establishes (1).

For (2), the proof follows that of \cite[Lemma 4.3]{BCIL} but needs a slight modification.   For two matrices $A$ and $B$ in $M_n(K_\infty)$, we write ``$A \preceq B$" if $\vert a_{ij}^\nu \vert \leq \vert b_{ij}^\nu \vert$ for all $1 \leq i, j \leq n$ and $\nu \mid \infty$, where $a_{ij}^\nu$ and $b_{ij}^\nu$ are the $(i,j)$-entry of $A^\nu$ and $B^\nu$ respectively.  Then \cite[Lemma 4.2]{BCIL} holds in the present setting and we may carry the proof of \cite[Lemma 4.3]{BCIL} over verbatim.
\end{proof}

At last, here is the definition of balanced HKZ-reduced forms.

\begin{defn} \label{balancedhkzdefn}
Let $K$ be a totally real number field.  A positive definite quadratic form $Q$ in $n$ variables over $K$ is called {\em balanced HKZ-reduced} if it is weakly reduced and  for $1 \leq i < j \leq n$ the $(i,j)$-th inner coefficients of both $Q$ and $Q^{-1}$ are lying inside $\mathcal E(j-i)$.
\end{defn}

\begin{prop} \label{balancedhkz}
If $K$ has class number 1, then every positive definite quadratic form over $K$ is integrally equivalent to a balanced HKZ-reduced quadratic form.
\end{prop}
\begin{proof}
This is clear.
\end{proof}

\section{Technical lemmas}

The following two lemmas address the representation of quadratic forms by $I_5$.  They are consequences of  \cite[Theorem 3]{HKK} and
Kneser's theory of neighbors of quadratic lattices \cite{Kn}, respectively.

\begin{lem} \label{rep15}
There exists a positive rational integer $r$, depending only on $K$, such that if $a$ is a totally positive element in $\O$ and $\TraceK(a) \geq r$ then $a$ is represented by $I_5$.
\end{lem}

\begin{lem} \label{rep25}
Let $\p$ be a prime ideal of $K$.  There exists a positive integer $\ell$, depending only on $\p$, such that a positive semidefinite integral quadratic form of rank $2$ over $K$ is represented by $I_5$ whenever its scale is divisible by $\p^\ell$.
\end{lem}

From now on, we fix a rational prime $p$ which is unramified in $K$ and let $\ell$ be a positive integer from Lemma \ref{rep25} for all prime ideal divisors of $p\O$.

\begin{lem} \textnormal{\cite[Lemma 6.1(a)]{BCIL}} \label{AplusS1}
Let $n \geq 2$ be a positive integer, $A = \diag(a_1, \ldots, a_n)$ be a diagonal matrix in $M_n(\R)$, and $S = (s_{ij})$ be a symmetric matrix in $M_n(\R)$.  If, for each $1 \leq i \leq n$, $a_i = \sum_{j=1}^n t_{ij}$ with $t_{ij} > 0$ and $t_{ij}t_{ji} > s_{ij}^2$ for all $j$, then $A + S$ is positive definite.
\end{lem}

\begin{cor} \textnormal{\cite[Corollary 6.2]{BCIL}} \label{smallS}
If $S = (s_{ij})$ is a symmetric matrix in $M_n(\R)$ with $\vert s_{ij} \vert <  \frac{1}{n}$ for all $i, j$, then $I + S$ is positive definite.
\end{cor}

Let $\{\omega_1, \ldots, \omega_d\}$ be an integral basis of $K$ with all the $\omega_i$ totally positive.  Let
$$\mathcal P := \left\{\sum_{i=1}^d f_i\omega_i : 1 \leq f_i \leq p^{\ell}  \mbox{ for all } i\right\},$$
which is a complete set of representatives of the cosets in $\O/p^\ell\O$.  All the elements in $\mathcal P$ are totally positive integers.

Let $\pi = f_1\omega_1 + \cdots + f_d\omega_d$ be an element in $\mathcal P$, and let $f_\pi = \sum_{i=1}^d f_i^2$ and $w_\pi: = \sum_{i=1}^d \omega_i^2$, both of which are totally positive integers in $K$.   For convenience, we will call $(f_\pi, w_\pi)$ ``the pair" associated with $\pi$.  It is clear that
\begin{equation} \label{pi}
\begin{pmatrix}
f_\pi & \pi \\
\pi & w_\pi
\end{pmatrix} \rep I_{d}.
\end{equation}
 Define
 \begin{equation} \label{gamma}
\gamma: = \left \lceil \max_{\pi \in \mathcal P} \left\{\lcrown f_\pi \rcrown, \lcrown w_\pi \rcrown, \lcrown \pi \rcrown \right\} \right \rceil,
\end{equation}
which is a positive rational integer depending only on the choice of the integral basis of $K$ and hence can be viewed as a constant depending only on $K$.

\begin{lem} \label{AplusS2}
Let $n \geq 2$ be a positive integer, $A = \diag(a_1, \ldots, a_n)$ be a diagonal matrix in $M_n(\O)$, $S = (s_{ij})$ be a symmetric matrix in $M_n(\O)$.  For each $s_{ij}$, let $\pi_{ij}$ be the element in $\mathcal P$ such that $s_{ij} \equiv \pi_{ij} \mod p^\ell\O$.  Suppose that for $1 \leq i \leq n$, $a_i = \sum_{j=1}^n t_{ij}$ with $t_{ij} \succ 0$ and $t_{ij}t_{ji} \succ (s_{ij} - \pi_{ij})^2$ for all $j$.  If $\lbase t_{ii} + s_{ii} \rbase > 2(n-1)\gamma + \frac{r}{d}$ for all $i$, then $A + S$ is represented by $I_{5n + \frac{n(n-1)}{2}(d + 5)}$.
\end{lem}
\begin{proof}
As is done in \eqref{pi}, for any $i < j$, let $(n_{ij}, n_{ij})$ be the pair associated with $\pi_{ij}$ such that
$$\begin{pmatrix}
n_{ij} & \pi_{ij} \\
\pi_{ji} & n_{ji}
\end{pmatrix} \rep I_d.$$

For all $i \neq j$, since $\gamma \succ n_{ij}$ we can write
$$t_{ij} + \gamma = n_{ij} + (p^\ell t_{ij}' - \delta_{ij})$$
with $t_{ij}' \succ 0$ and $\delta_{ij} \in \mathcal P$.  Then
$$a_i + s_{ii} = b_i + \sum_{j\neq i} n_{ij} + \sum_{j\neq i} p^\ell t_{ij}'$$
where
$$b_i = t_{ii} + s_{ii} - \left((n-1)\gamma + \sum_{j\neq i} \delta_{ij}\right).$$
Since $\lbase t_{ii} + s_{ii} \rbase > 2(n-1)\gamma + \frac{r}{d}$ and $\gamma \succ \delta_{ij}$ for all $j \neq i$, we have
$$\TraceK(b_i) = \sum_{\nu\mid \infty} b_i^\pnu > d\cdot \frac{r}{d} = r$$
which implies that each $b_i$ is represented by $I_5$ by Lemma \ref{rep15}.

Let $p^\ell q_{ij} = s_{ij} - \pi_{ij}$ for all $i, j$.  Now we can write
\begin{equation*}
\begin{split}
A+S&=\sum_ia_iE_{ii}+\sum_{1\leq i,j\leq n}s_{ij}E_{ij}\\
   &=\sum_i(t_{ii}+s_{ii})E_{ii}+\sum_{j\neq i}t_{ij}E_{ii}+\sum_{j\neq i}p^{\ell}q_{ij}E_{ij}+\sum_{j\neq i}\pi_{ij}E_{ij}\\
   &=\sum_ib_iE_{ii}+\sum_{j\neq i}p^{\ell}t'_{ij}E_{ii}+\sum_{j\neq i}p^{\ell}q_{ij}E_{ij}+\sum_{j\neq i}n_{ij}E_{ii}+\sum_{j\neq i}\pi_{ij}E_{ij}\\
   &=\mathrm{diag}(b_1,...,b_n)+\sum_{i<j}(p^\ell t'_{ij}E_{ii}+ p^\ell q_{ij}E_{ij} + p^\ell q_{ji}E_{ji} + p^\ell t_{ji}'E_{jj}) \\
    & \quad  + \sum_{i < j} (n_{ij}E_{ii} + \pi_{ij}E_{ij} + \pi_{ji}E_{ji} + n_{ji}E_{jj}).\nonumber
\end{split}
\end{equation*}
The diagonal matrix $\diag(b_1, \ldots, b_n)$ is represented by $I_{5n}$.  Each of the $\frac{n(n-1)}{2}$ symmetric matrices $(n_{ij}E_{ii} + \pi_{ij}E_{ij} + \pi_{ji}E_{ji} + n_{ji}E_{jj})$ is represented by $I_d$ by \eqref{pi}.  For any $j \neq i$, $p^\ell t_{ij}' \succ r_{ij} = t_{ij} + \gamma - n_{ij} \succ t_{ij}$ and $t_{ij}t_{ji} \succ (s_{ij} - \pi_{ij})^2 = (p^\ell q_{ij})^2$.  Thus, the symmetric matrix $(p^\ell t'_{ij}E_{ii}+ p^\ell q_{ij}E_{ij} + p^\ell q_{ji}E_{ji} + p^\ell t_{ji}'E_{jj})$ is positive semidefinite of rank 2 and with scale divisible by $p^\ell$, hence it is represented by $I_5$ by Lemma \ref{rep25}.  Therefore, $A + S$ is represented by $I_{5n + \frac{n(n-1)}{2}(d + 5)}$.
\end{proof}

\section{Proof of main theorem}

The following proposition is the backbone of the proof of Theorem \ref{maintheorem}.

\begin{prop} \label{backbone}
Suppose that $K$ has class number $1$.  There exist constants $\xi$ and $D_5$, depending only on $K$, such that if the minimum of a positive definite integral quadratic form over $K$ in $n\,(\geq 2)$ variables is larger than $D_5\,e^{\xi\sqrt{n}}$ then it is represented by $I_{6n + \frac{n(n-1)}{2}(d + 5)}$.
\end{prop}
\begin{proof}
Let $Q$ be a positive definite integral quadratic form over $K$ in $n$ ($\geq 2$) variables.  We may assume that $Q$ is already balanced HKZ-reduced.  We will write $Q$ as a sum $P^tP + A + S$ of matrices in $M_n(\O)$, where $P$, $A$, and $S$ are matrices in $M_n(\O)$ with $A$ diagonal and $S$ symmetric.  With Lemma \ref{AplusS2} we will show that if $\min(Q)$ is large enough then $A + S$ is represented by $I_{5n + \frac{n(n-1)}{2}(d + 5)}$.  The proposition follows immediately as $P^tP$ is clearly represented by $I_n$.

Since $Q$ is balanced HKZ-reduced, we may write
$$Q(x_1, \ldots, x_n) = \sum_{i=1}^n h_i\left(x_i + \sum_{j=i+1}^n u_{ij}x_j\right)^{\!\!\!2},$$
with the outer coefficients $h_1, \ldots, h_n$ and inner coefficients $u_{ij}$ satisfying the inequalities and conditions in Definitions \ref{weakly-reduced-definition} and  \ref{balancedhkzdefn}.  The Gram matrix of $Q$ is $U^tHU$ with $H = \diag(h_1, \ldots, h_n)$ and $U$ upper triangular unipotent.
\vskip 2mm

\noindent {\bf Step 1.}  Here we will find a diagonal matrix $A$, with diagonal entries comparable with the outer coefficients of $Q$, such that $Q - A$ remains positive definite.

For each $\nu \mid \infty$ and $k = 1, \ldots, n$, let $\eta_{\nu k}$ be the real number such that
$$n^3 \lambda^2 \,\od(n)^{\frac{1}{d}}\,\oc(n)^2\, \eta_{\nu k} = h_k^\pnu,$$
where $\od$ and $\lambda$ are the function and the constant, respectively, appeared in Definition \ref{weakly-reduced-definition}, and $\oc$ is the function defined in  Lemma \ref{cm}.
Suppose that
\begin{equation} \label{step1}
\min(Q) > \od(n)\lambda^{d-1}\left(2\beta\lambda^2 n^3\, \od(n)^{\frac{1}{d}}\, \oc(n)^2\right)^d.
\end{equation}
Then, by conditions (1) and (2) in Definition \ref{weakly-reduced-definition}, for any $\nu \mid \infty$,
$$\lambda^{d-1} (h_k^\pnu)^d \geq \NormK(h_k) \geq \frac{\min(Q)}{\od(n)} > \lambda^{d-1} \left(2\beta\lambda^2 n^3\, \od(n)^{\frac{1}{d}}\, \oc(n)^2\right)^d.$$
Hence
$$h_k^\pnu > 2\beta\lambda^2 n^3\, \od(n)^{\frac{1}{d}}\, \oc(n)^2$$
and $\eta_{\nu k} > 2\beta$ as a result.  Now, by the definition of $\beta$, there exists $a_k \in \O$ such that
$$\vert a_k - (\eta_{\nu k} - \beta) \vert_\nu < \beta, \quad \mbox{ for all } \nu \mid \infty.$$
Then $\eta_{\nu k} - 2\beta < a_k^\pnu < \eta_{\nu k}$ and hence $a_k \succ 0$ for all $\nu \mid \infty$.  Let
$$A = \diag(na_1, \ldots, na_n).$$

Let $\sqrt{H}$ be the positive definite square-root of $H$ as a Humbert form.  Then $Q = I_n[\sqrt{H}U]$, and
$$Q - A = \left(I_n - A[U^{-1}\sqrt{H}^{-1}]\right)[\sqrt{H}U].$$
So, it suffices to show that $I_n - A[U^{-1}\sqrt{H}]$ is positive definite.  Let $y_{ij}$ and $b_{ij}$ be the $(i,j)$-entries of $U^{-1}$ and $A[U^{-1}]$, respectively.  Since $Q$ is balanced HKZ-reduced, $\vert y_{ij} \vert_\nu \leq \oc(n- i)$ for $1 \leq i \leq j \leq n$. Then,  for all $\nu \mid \infty$ and $1 \leq i \leq j \leq n$,
\begin{eqnarray*}
\vert b_{ij} \vert_\nu & \leq & \sum_{k=1}^{i} \vert na_k\, y_{ki}\, y_{kj} \vert_\nu \\
    & <  & \sum_{k=1}^{i} n^{-2}\,\od(n)^{-\frac{1}{d}}\, \lambda^{-2}\, \oc(n)^{-2} h_k^\pnu \,\, \oc(n-k)^2\\
    & \leq & \frac{1}{n^2} \sum_{k=1}^{i} h_k^\pnu \, \frac{1}{\od(n)^{\frac{1}{d}}\lambda^2}.
\end{eqnarray*}
By Corollary \ref{hihj3}, we have
$$\sqrt{h_k^\pnu} \leq \lambda\, \od(n)^{\frac{1}{2d}}\, \sqrt{h_\ell^\pnu}$$
for $\ell = i$ or $j$.  Therefore,
$$\vert b_{ij} \vert_\nu \leq  \frac{1}{n}\, \sqrt{h_i^\pnu h_j^\pnu},$$
and hence the $(i,j)$-entry of the $\nu$-th component of the Humbert form $A[U^{-1}\sqrt{H^{-1}}]$ is
$$\vert b_{ij} \vert_\nu \, \sqrt{h_i^\pnu h_j^\pnu}^{-1} < \frac{1}{n}.$$
Thus, by Corollary \ref{smallS}, $I - A[U^{-1}\sqrt{H^{-1}}]$ is positive definite.
\vskip 2mm

\noindent {\bf Step 2.} We now decompose $Q$ as $Q = P^tP + A + S$, a sum of symmetric matrices in $M_n(\O)$, and then estimate the size of the entries of $S$.

By the Gram-Schmidt process, there exists an upper triangular matrix $N$ in $M_n(K_\infty)$ such that
$$I_n - A[U^{-1}\sqrt{H}^{-1}] = N^tN.$$
For each $\nu \mid \infty$, let $n_{ij}^\nu$ denote the $(i,j)$-entry of the $\nu$-th component of $N$.  The Humbert form $I_n - N^tN = A[U^{-1}\sqrt{H}^{-1}]$ is positive definite; so at each $\nu \mid \infty$ and $1 \leq j \leq n$,
$$1 - \sum_{i \leq j} \vert n_{ij}^\nu\vert_\nu^2 > 0$$
and hence $\vert n_{ij}^\nu \vert_\nu < 1$.

Let $W$ be the upper triangular matrix $N\sqrt{H} U$ in $M_n(K_\infty)$. Then
$$W^tW = N^tN[\sqrt{H}U] = Q - A,$$
For $j \geq i$, let $w_{ij}^\nu$ denote the $(i,j)$-entry of the $\nu$-th component of $W$.  Then,
\begin{eqnarray*}
\vert w_{ij}^\nu\vert_\nu &\leq & \sum_{k=i}^j \vert n_{ij}^\nu \vert_\nu \,\sqrt{h_k^\pnu}\,\, \vert u_{kj} \vert_\nu\\
    & < & \sum_{k=i}^j \sqrt{h_k^\pnu}\,\, \vert u_{kj} \vert_\nu\\
    & \leq & n\, \lambda\, \sqrt{\od(n)^{1/d}\, h_j^\pnu}\,\, \oc(j).
\end{eqnarray*}
By the definition of $\beta$, there exists an upper triangular matrix $P$ in $M_n(\O)$ such that the entries of $F: = W - P$ are in ${\mathcal B}_K$.  Note that $F$ is also upper triangular.  For $\nu \mid \infty$, let $f_{ij}^\nu$ denote the $(i,j)$-entry of the $\nu$-th component of $F$.  A simple algebraic manipulation shows that $Q = P^tP + A + S$, where
$$S = F^tW + W^tF - F^tF$$
which is a symmetric matrix in $M_n(\O)$.  Therefore, for $j \geq i$, the size of the $\nu$-th component of the $(i,j)$-entry of $F^tW$ is
$$\left\vert \sum_{k=1}^{j} f_{ki}^\nu\, w_{kj}^\nu \right \vert_\nu \leq n\, \beta\, n\, \lambda \sqrt{\od(n)^{1/d} h_j^\pnu}\,\,\oc(j).$$
So,
\begin{equation} \label{sij}
\vert s_{ij} \vert_\nu \leq 2n^2\, \beta\, \oc(j)\, \lambda\, \sqrt{\od(n)^{1/d}h_j^\pnu} + n\beta^2.
\end{equation}

Suppose that
\begin{equation} \label{step2}
\min(Q) \geq \left(\frac{n\beta^2 + \gamma}{n^2\beta}\right)^{2d}\lambda^{d-1}.
\end{equation}
where $\gamma$ is defined by \eqref{gamma}.

Since $\lambda^{d-1} (h_1^\pnu)^d \geq \NormK(h_1) = \min(Q)$, \eqref{step2} implies that
\begin{eqnarray*}
n^2\, \beta\, \oc(j)\, \lambda\, \sqrt{\od(n)^{1/d}h_j^\pnu} - \gamma & \geq &  n^2\, \beta \, \oc(j)\, \sqrt{h_1^\pnu} - \gamma \\
        & \geq & n^2\, \beta\, \oc(j)\frac{n\beta^2 + \gamma}{n^2\beta} - \gamma \\
        & \geq & n\beta^2 + \gamma - \gamma \\
        & \geq & n\beta^2.
\end{eqnarray*}
Combining with \eqref{sij}, we have
$$\vert s_{ij}\vert_\nu \leq 3n^2\, \beta\, \oc(j)\, \lambda\, \sqrt{\od(n)^{1/d}h_j^\pnu} - \gamma,$$
and
$$\vert s_{ij} - \pi_{ij} \vert_\nu \leq \vert s_{ij} \vert_\nu + \vert \pi_{ij} \vert_\nu \leq  3n^2\, \beta\, \oc(j)\, \lambda\, \sqrt{\od(n)^{1/d}h_j^\pnu}.$$
\vskip 2mm

\noindent {\bf Step 3.} We will show that $A + S$ is represented by a sum of squares provided $\min(Q)$ is larger than $D_5\, e^{\xi\sqrt{n}}$ for some constants $D_5$ and $\xi$ depending only on $K$.

For $1 \leq i, j \leq n$, let $t_{ij} = a_i$ so that $na_i = \sum_{j=1}^n t_{ij}$.    Suppose that
\begin{equation} \label{step3a}
\min(Q) \geq \lambda^{d-1}\, \od(n) \, \left(2n^3\, \lambda^2\, \od(n)^{1/d}\, \oc(n)^2\, 2\beta\right)^d.
\end{equation}
Since $\od(n)\, \lambda^{d-1} (h_i^\pnu)^d \geq \min(Q)$ for each $\nu \mid \infty$, we have
$$h_i^\pnu \geq  2n^3\, \lambda^2\, \od(n)^{1/d}\, \oc(n)^2\, 2\beta$$
and hence
\begin{eqnarray*}
t_{ij}^\pnu = a_i^\pnu & > & \eta_{\nu i} - 2\beta \\
    & = & \frac{h_i^\pnu}{n^3\, \lambda^2\, \od(n)^{1/d}\, \oc(n)^2} - 2\beta \\
    & \geq & \frac{h_i^\pnu}{2n^3\, \lambda^2\, \od(n)^{1/d}\, \oc(n)^2}.
\end{eqnarray*}
Thus,
$$t_{ij}^\pnu t_{ji}^\pnu \geq \frac{h_i^\pnu h_j^\pnu}{4n^6\, \lambda^4\, \od(n)^{2/d}\, \oc(n)^4}.$$
But
$$\vert s_{ij} - \pi_{ij} \vert_\nu^2 \leq 9n^4\, \beta^2\, \oc(n)^2\, \lambda^2\, \od(n)^{1/d}\, h_j^\pnu$$
so, for $i \leq j$, we have
\begin{eqnarray*}
t_{ij}^\pnu t_{ji}^\pnu & \geq & \vert s_{ij} - \pi_{ij} \vert_\nu^2 \frac{h_i^\pnu}{36n^{10}\, \beta^2\, \lambda^6\, \od(n)^{3/d}\, \oc(n)^6}\\
    & = & 4\vert s_{ij} - \pi_{ij} \vert_\nu^2\frac{\lambda^2\, \od(n)^{1/d}\, h_i^\pnu}{144 n^{10}\, \beta^2\, \lambda^8\, \od(n)^{4/d}\, \oc(n)^6}\\
    & \geq & 4 \vert s_{ij} - \pi_{ij} \vert_\nu^2 \frac{h_1^\pnu}{144 n^{10}\, \beta^2 \, \lambda^8 \, \od(n)^{4/d}\, \oc(n)^6}.
\end{eqnarray*}

Suppose that
\begin{equation} \label{step3b}
\min(Q) \geq \lambda^{d-1}\, \left(144n^{10}\, \beta^2\, \lambda^8\, \od(n)^{4/d}\, \oc(n)^6\right)^d.
\end{equation}
Then $h_1^\pnu \geq 144n^{10}\, \beta^2\, \lambda^8 \, \od(n)^{4/d}\, \oc(n)^6$.  So,
$$t_{ij}^\pnu t_{ji}^\pnu \geq 4\vert s_{ij} - \pi_{ij} \vert_\nu^2 \geq \vert s_{ij} - \pi_{ij} \vert_\nu^2.$$
As a result, $t_{ij}t_{ji} \succ (s_{ij} - \pi_{ij})^2$ and $t_{ii}^\pnu > 2\vert s_{ii} - \pi_{ii} \vert_\nu$.  Therefore,
\begin{eqnarray*}
t_{ii}^\pnu + s_{ii}^\pnu =  t_{ii}^\pnu + s_{ii}^\pnu - \pi_{ii}^\pnu + \pi_{ii}^\pnu & \geq & \frac{1}{2}t_{ii}^\pnu + \pi_{ii}^\pnu\\
    & \geq & \frac{1}{2}t_{ii}^\pnu\\
    & = & \frac{1}{2}a_i^\pnu\\
    & > & \frac{1}{2}\eta_{\nu i} - \beta.
\end{eqnarray*}
From the definition of $\eta_{\nu i}$ we have
$$\eta_{\nu i} = \frac{h_i^\pnu}{n^3 \,\lambda^2\, \od(n)^{1/d}\, \oc(n)^2} \geq \frac{h_1^\pnu}{n^3\, \lambda^4\, \od(n)^{2/d}\, \oc(n)^2}.$$
Therefore,
$$t_{ii}^\pnu + s_{ii}^\pnu > \frac{h_1^\pnu}{2n^3\, \lambda^4\, \od(n)^{2/d}\, \oc(n)^2} - \beta.$$

Suppose that
\begin{equation} \label{step3c}
\min(Q) \geq \lambda^{d-1}\left(2n^3 \, \lambda^4\, \od(n)^{2/d}\, \oc(n)^2\, \left(4(n-1)\gamma + \frac{r}{d} + \beta\right)\right)^d,
\end{equation}
where $r$ is the rational integer obtained by Lemma \ref{rep15}. Then, since we have the inequality $\lambda^{d-1}(h_1^\pnu)^d \geq \min(Q)$,
$$\lbase t_{ii} + s_{ii} \rbase > 4(n-1)\gamma + \frac{r}{d}.$$

It is easy to see that there exist constants $D_5$ and $\xi$, depending only on $K$, such that if $\min(Q) > D_5\,e^{\xi\sqrt{n}}$ then \eqref{step1}, \eqref{step2}, \eqref{step3a}, \eqref{step3b}, and \eqref{step3c} are all satisfied and hence $A + S$ is represented by $I_{5n + \frac{n(n-1)}{2}(d + 5)}$ by Lemma \ref{AplusS2}.
\end{proof}

\begin{proof}[Proof of Theorem \ref{maintheorem}]
Let $Q(x_1, \ldots, x_n)$ be a positive definite integral quadratic form over $K$ which is represented by a sum of squares, say,
$$Q(x_1, \ldots, x_n) = \sum_{i=1}^N (a_{i1}x_1 + \cdots a_{in}x_n)^2$$
where $a_{ij} \in \O$ for all $i, j$.   By picking a larger $D_5$ in Proposition \ref{backbone} if necessary, we may assume that $g_\O(1) \leq D_5\,e^{\xi}$.  So, we may further assume that $n \geq 2$.   It follows from Proposition \ref{backbone} that if $\min(Q) > D_5\,e^{\xi\sqrt{n}}$, then $Q$ is represented by a sum of at most $6n + \frac{n(n-1)}{2}(d + 5)$ squares.

Now, we suppose that $\min(Q) \leq D_5\,e^{\xi\sqrt{n}}$.  By switching $Q$ to an integrally equivalent form that is balanced reduced, we have the equality $\min(Q) = \NormK(a_{11}^2 + \cdots + a_{N1}^2)$.   Assume without loss of generality that $a_{11}, \ldots, a_{M1}$ are nonzero and the rest of the $a_{i1}$'s are zero.  Then,
$$M \leq \sum_{i=1}^M \NormK(a_{i1}^2) \leq  \NormK(a_{11}^2 + \cdots + a_{M1}^2) = \min(Q).$$
So, $M \leq \lfloor D_5\,e^{\xi\sqrt{n}} \rfloor$, and hence
$$Q(x_1, \ldots, x_n) = \sum_{i=1}^{\lfloor D_5\,e^{\xi\sqrt{n}} \rfloor} (a_{i1}x_1 + \cdots a_{in}x_n)^2 + Q'(x_2, \ldots, x_n),$$
and $Q'$ is a positive semidefinite quadratic form in $n-1$ variables represented by a sum of squares.  By definition, $Q'$ must be represented by the sum of $g_\O(n-1)$ squares.  Thus,
$$g_\O(n) \leq \lfloor D_5\,e^{\xi\sqrt{n}} \rfloor + g_\O(n-1).$$
Inductively, we see that $g_\O(n) \leq D_5\, n\, e^{\xi\sqrt{n}} \leq D\,e^{\kappa\sqrt{n}}$ for some suitable larger constants $D$ and $\kappa$.
\end{proof}

\end{document}